\documentclass{amsart}
\usepackage{layout}
\usepackage{amssymb, amsfonts,amsthm,amsmath,latexsym}
\usepackage[all]{xy}

\def\bt{\begin{thm}}
\def\et{\end{thm}}
\def\bl{\begin{lem}}
\def\el{\end{lem}}
\def\bd{\begin{defi}}
\def\ed{\end{defi}}
\def\bc{\begin{cor}}
\def\ec{\end{cor}}
\def\bp{\begin{proof}}
\def\ep{\end{proof}}
\def\br{\begin{rem}}
\def\er{\end{rem}}

\newtheorem{thm}{Theorem}[section]
\newtheorem{prop}[thm]{Proposition}
\newtheorem{lem}[thm]{Lemma}
\newtheorem{defn}[thm]{Definition}

\newtheorem{rem}[thm]{Remark}
\newtheorem{cor}[thm]{Corollary}
\newtheorem{question}[thm]{Question}

\numberwithin{equation}{section}

\newcommand{\fraction}[2]{\frac{\textstyle #1}{\textstyle #2}}
\newcommand{\vmin}{V_{\theta}}
\newcommand{\Tmin}{T_{\theta}}
\newcommand{\cohomology}{H^{1,1}(X,\mathbb{R})}
\newcommand{\cohomologyp}{H^{p,p}(X,\mathbb{R})}
\newcommand{\psef}{H^{1,1}_{psef}(X,\mathbb{R})}
\newcommand{\nef}{H^{1,1}_{nef}(X,\mathbb{R})}
\newcommand{\bigc}{H^{1,1}_{big}(X,\mathbb{R})}

\newcommand{\energy}{\mathcal{E}(X,\theta)}
\newcommand{\energyc}{\mathcal{E}_{\chi}(X,\theta)}
\newcommand{\energyp}{\mathcal{E}^p(X,\theta)}
\newcommand{\energypc}{\mathcal{E}^p(\alpha)}
\newcommand{\psh}{Psh(X,\theta)}

\newcommand{\jac}{Jac_{\omega}(f)}
\newcommand{\del}{\partial}

\newcommand{\bthm}{\begin{thm}}
\newcommand{\ethm}{\end{thm}}
\newcommand{\bstp}{\begin{stp}}
\newcommand{\estp}{\end{stp}}
\newcommand{\blemma}{\begin{lemma}}
\newcommand{\elemma}{\end{lemma}}
\newcommand{\bprop}{\begin{prop}}
\newcommand{\eprop}{\end{prop}}
\newcommand{\bpf}{\begin{pf}}
\newcommand{\epf}{\end{pf}}
\newcommand{\bdefn}{\begin{defn}}
\newcommand{\edefn}{\end{defn}}
\newcommand{\brk}{\begin{rmrk}}
\newcommand{\erk}{\end{rmrk}}
\newcommand{\bcrl}{\begin{crl}}
\newcommand{\ecrl}{\end{crl}}

\title[Equidistribution towards the Green Current]{Equidistribution towards the Green Current\\ in Big Cohomology Classes}
\author{Turgay Bayraktar}
\date{\today}
\address{Mathematics Department, Johns Hopkins University Baltimore,  21218 MD USA}
\email{bayraktar@jhu.edu}
\begin{document}

\begin{abstract}
We study limiting distribution of the sequence of pull-backs of smooth $(1,1)$ forms and positive closed currents %and pre-images of subvarieties 
by meromorphic self-maps of compact K\"ahler manifolds.
\end{abstract}
%\begin{abstract}
%Let $f:X\to X$ be a dominant meromorphic map of a compact K\"ahler manifold and $\lambda>1$ such that $f^*\alpha=\lambda\alpha$ for some $\alpha\in \psef.$ Under natural dynamical assumptions on $f$, we provide sufficient conditions on positive closed currents in $\alpha$ whose pull-backs converge to the Green current associated to the class $\alpha.$%We consider the asymptotic pre-images of positive closed $(1,1)$ currents under the pull-back operator induced by iterates of $f.$ We provide sufficient conditions on 
%We study the equidistribution problem in codimension one for meromorphic maps of compact K\"ahler manifolds. 
%\end{abstract}

\maketitle
%\tableofcontents
\section{Introduction}
Let $f:\mathbb{P}^1\to \mathbb{P}^1$ be a rational map of degree $\lambda\geq 2$ and $\nu$ be any probability measure on $\mathbb{P}^1.$ A theorem of Brolin \cite{Br}, Lyubich \cite{Lu} and Freire-Lopez-Ma{\~n}{\'e} \cite{FLM} asserts that the sequence of pre-images $\lambda^{-n}(f^n)^*\nu$ converges weakly to the measure of maximal entropy $\mu_f$ if and only if $\nu(\mathcal{E}_f)=0$ where $\mathcal{E}_f$ is an (possibly empty) exceptional set which contains at most two points. Recall that \textit{Green current} $T_f$ for a holomorphic map $f:\Bbb{P}^k \to \Bbb{P}^k$ of degree $\lambda\geq 2$ is defined to be the weak limit of the sequence $\lambda^{-n}(f^n)^*\omega_{FS}$ where $\omega_{FS}$ denotes the Fubini-Study form on $\Bbb{P}^k.$ In \cite{FJ,FJ07} Favre and Jonsson provided a complete characterization of the exceptional set for holomorphic endomorphisms of $\mathbb{P}^2$ stating that there exists an exceptional set $\mathcal{E}_f$ containing at most three lines and finitely many points which are totally invariant such that $\lambda^{-n}(f^n)^*S$ converges weakly to the Green current $T_f$ if and only if the positive closed current $S$ does not put any mass on the components of $\mathcal{E}_f$.  The papers \cite{DS08,Pa} also give sufficient conditions for similar results in higher dimensions. More generally, equidistribution problem has been studied for various category of maps within the last two decades (see for instance \cite{FS,RS,G03,G04,DS06,B}). However, unlike the rational maps of $\Bbb{P}^1,$ for bidegree $(1,1)$ currents there is no general existence theorem for the Green currents nor a complete characterization of the exceptional set of positive closed currents on compact K\"ahler manifolds.\\ \indent   %In the case of rational endomorphisms of $\pk$ various equidistribution results obtained in \cite{FS,RS,G03,G04,DS06}.\\ \indent %proved the existence of a finite collection $\mathcal{E}_{DS}$ of totally irreducible algebraic subsets of $\pk$ such that if $S$ does not charge any component of $\mathcal{E}_{DS}$ then $d^{-n}(f^n)^*S$ converges weakly to the Green current $T_f.$ Later, Parra \cite{Pa} obtained finer results in this direction.     
 %In the case of algebraically stable rational endomorphisms of $\Bbb{P}^k,$ Russakovskii and Shiffman \cite{RS} proved that there exists a pluripolar set $\mathcal{E}\subset (\Bbb{P}^k)^*$ such that for every hyperplane $H \not\in \mathcal{E}$ $$\lambda^{-n}(f^n)^*([H]-\omega_{FS})\to 0$$
 %in the sense of currents, %equidistributes towards the Green current $T_f$ 
   %Similar results were obtained in \cite{DS06} when the ambient space is an arbitrary projective manifold. On the other hand, Guedj \cite{G03} showed that for any positive closed $(1,1)$ current $S$ on $\Bbb{P}^k$ with identically zero Lelong numbers $\lambda^{-n}(f^n)^*S$ converges weakly to the Green current $T_f.$\\ \indent
 The purpose of this work is to study the limiting distribution of smooth bidegree $(1,1)$ forms and positive closed currents under the pull-back operator induced by meromorphic endomorphisms of compact K\"ahler manifolds.\\ \indent  %prove analogous results in the setting of meromorphic mappings of compact K\"ahler manifolds.
 Our setting as follows. Let $X$ be a compact K\"ahler manifold of dimension $k$ and $f:X\to X$ be a dominant meromorphic map. Then $f$ induces a linear map $f^*:\cohomology \to~\cohomology$ which is, in general, not compatible with the dynamics of $f.$  We say that $f$ is \textit{1-regular} if $(f^n)^*=(f^*)^n$ for each $n=1,2\dots$ In the sequel, we assume that $f$ is 1-regular and $\lambda_1(f)>1$ where $\lambda_1(f)$ denotes the spectral radius of $f^*.$ Recall that a class $\alpha \in \cohomology$ is called pseudo-effective (psef) if it can be represented by a positive closed current, we say that $\alpha$ is big if it can be represented by a strictly positive current. The set of psef classes forms a salient closed convex cone. It follows from a Perron-Frobenius type argument that there exists a class $\alpha \in \psef$ such that $f^*\alpha=\lambda_1(f)\alpha$. In general, the class $\alpha$ can not be represented by a smooth positive form (see \cite{B} for examples of such  mappings). %It follows that under a natural dynamical assumption (see Theorem \ref{B}) the sequence $\frac{1}{\lambda_1(f)^{n}}(f^n)^*\theta$ converges to a positive closed current $T_{\alpha}$ for every smooth representative $\theta$ of $\alpha.$  
  \begin{thm}\label{A}
Let $f:X\to X$ be a 1-regular dominant meromorphic map such that $\lambda:=\lambda_1(f)>1$ is a simple eigenvalue of $f^*$ with $f^*\alpha=\lambda\alpha$ . If $\alpha \in \psef$ can be represented by a positive current with identically zero Lelong numbers then there exists a positive closed current $T_{\alpha}\in \alpha$ such that for every smooth form $\theta\in \alpha$ $$\frac{1}{\lambda^n}(f^n)^*\theta \to T_{\alpha}$$ in the sense of currents as $n\to\infty.$ Moreover, for every positive closed current $S\in \alpha$ with identically zero Lelong numbers on $X$ $$\frac{1}{\lambda^n}(f^n)^*S \to T_{\alpha}$$ in the sense of currents as $n\to\infty$.
\end{thm}
The current $T_{\alpha}$ is called \textit{Green current} associated to the class $\alpha$. Recall that any $\alpha\in \bigc\cap \nef$ can be represented by a positive closed current with identically zero Lelong numbers on $X$ (\cite{DPS01,G04}). If the invariant class is merely big we let $\energypc$ denote the energy classes (see section \ref{ener} for details) of positive closed currents whose de Rham class is $\alpha$. 
\begin{prop}
Let $(X,f)$ be as in Theorem \ref{A}. If $\alpha \in \bigc$ then for every $S \in \energypc$ with $p>\dim X-1$
$$\frac{1}{\lambda^n}(f^n)^*S \to T_{\alpha}$$ in the sense of currents.
\end{prop}

%In the case of minimally separating maps (see section \ref{min} for the definition), we have the following:
%\begin{thm}
%Let $f:X\to X$ be a minimally separating bimeromorphic map. Assume that $\lambda>1$ and simple and $\alpha \in \bigc.$ If $\mathcal{I_1^-} \subset Amp(\alpha)$ then for every smooth closed $(1,1)$ form $\eta$ on $X$ 
%$$\frac{1}{\lambda^n}(f^n)^*\eta \to cT_{\alpha}$$ where $c>0$ depends only on $\{\eta\}$. Moreover, if the ambient space $X$ is projective then $$\frac{1}{\lambda^n}(f^n)^*[H] \to T_{\alpha}$$ for Lebesgue almost every hyperplane section $H.$
%\end{thm}

Finally, %we investigate the relationship between cohomological properties of  $f^*$-invariant classes and the ambient space $X.$ A theorem of Diller and Favre \cite{DF} asserts that a bimeromorphic map $f$ of a compact K\"ahler surface with $\lambda_1(f)>1$ is bimeromorphically conjugate to an automorphism if and only if the invariant class is not big. We provide a family of birational maps $f_{\tau}$ of a rational threefold $X_{\tau}$ such that $f_{\tau}$ is 1-regular, $\lambda_1(f_{\tau})>1$ is simple and the corresponding normalized invariant class is not big. However, the map $f_{\tau}$ is not bimeromorphically conjugate to an automorphism.\\ \indent  Finally, 
we consider the problem of which meromorphic maps admit a big invariant class. We prove that in the case of dimension two, only rational maps have this property. This improves the corresponding result of \cite{DF} which was proved in the case of bimeromorphic maps.
\begin{thm}\label{class}
Let $X$ be a compact K\"ahler surface and $f:X\to X$ be a dominant 1-regular meromorphic map. If $\lambda_1(f)>d_{top}(f)$ and there exists $\alpha\in \bigc$ such that $f^*\alpha=\lambda_1(f)\alpha$ then $X$ is a rational surface.
\end{thm}
The outline of the paper as follows: in section 2 we provide background, in section 3 we study contraction of volume of a set under the iteration by a meromorphic endomorphism, in section 4 and 5 we focus on equidistribution problem for various category of maps. For instance, in the algebraic case if the expanded class is the first Chern class of a finitely generated big line bundle then we prove that Green current describes the limiting distribution of pre-images of zero divisor of generic global holomorphic section (Theorem \ref{finite}). Finally, in section 6 we prove Theorem \ref{class}.

\section{Preliminaries}\label{prelim}
Let $X$ be a connected compact K\"ahler manifold with $\dim_{\mathbb{C}}X=k.$ We fix a K\"ahler form $\omega$ on $X$ such that $\int_X \omega^k=1.$ All volumes will be computed with respect to the probability volume form $dV:=\omega^k$. Let $H^{1,1}(X)$ denote the Dolbeault cohomology group and let $H^2(X,\mathbb{Z})$, $H^2(X,\mathbb{R})$ and $H^2(X,\mathbb{C})$ denote the de Rham cohomology groups with coefficients in $\mathbb{Z}, \mathbb{R}, \mathbb{C}.$ We also set $$H^{1,1}(X,\mathbb{R}):=H^{1,1}(X) \cap H^2(X,\mathbb{R}).$$ 
We say that a cohomology class $\alpha$ is \textit{pseudo-effective} (psef) if $\alpha$ ca be represented by a positive closed $(1,1)$ current. We say that $\alpha$ is \textit{big} if there exists a \textit{K\"ahler current} in $\alpha$ that is there exists $T_+\in \alpha$ such that $T_+\geq \epsilon\omega$ for some small $\epsilon>0$. By Demailly's approximation theorem \cite{Dem} we may assume that $T_+$ has analytic singularities that is locally $T_+$ can be written as $$T_+=dd^c(u+\frac{c}{2}\log \sum_{j=1}^N|f_j|^2)$$ where $d=\partial+\bar{\partial},\ d^c=\frac{i}{2\pi}(\bar{\partial}-\partial),\ c>0$, $u$ is a smooth function and $f_j's$ are local holomorphic functions. Thus, $T_+$ is smooth on a Zariski open set. The \textit{Ample locus}, $Amp(\alpha)$ of  a big class $\alpha$ is defined to be the largest such Zariski open subset. In fact, by the Noetherian property of analytic subsets there exists a K\"ahler current $T_K \in \alpha$ with analytic singularities such that $Amp(\alpha)=X\setminus E_+(T_K)$ where $E_+(T_K)=\{x\in X: \nu(T_K,x)>0\}$ and $\nu(T_K,x)$ denotes the Lelong number of $T_K$ at $x\in X$ \cite{Bou}. Roughly speaking, the set $Amp(\alpha)$ is the largest Zariski open set where the class $\alpha$ behaves like a K\"ahler class. In fact, $\alpha$ is K\"ahler if and only if $Amp(\alpha)=\emptyset$ (see \cite{Bou} for details). The set of psef classes, $\psef$ is a salient closed convex cone. The set of big classes, $\bigc$ is an open convex cone which coincides with the interior of $\psef.$ %We stress that these notions coincide with the classical ones in algebraic geometry if the ambient space is projective and $\alpha=c_1(L)$ is the first Chern class of a holomorphic line bundle \cite{Dem}.\\ \indent

Let $\alpha \in \psef$ and $\theta$ be a fixed smooth representative of $\alpha.$ An upper semi continuous function $\varphi \in L^1(X)$ is called \textit{$\theta$-plurisubharmonic} ($\theta$-psh) if $\theta + dd^c \varphi \geq 0$ in the sense of currents. Notice that a $\theta$-psh function is locally difference of a psh function and a smooth function. We denote the set of all $\theta$-psh functions by $\psh$. We say that a function $\varphi$ is \textit{quassi-plurisubharmonic} (qpsh) if $\varphi$ is $\theta$-psh for some closed smooth $(1,1)$ form $\theta$ on $X.$ Let $\alpha\in \psef,$  $T \in \alpha$ be a positive closed current and $\theta$ be a smooth representative  of $\alpha$, by $dd^c$ lemma \cite{GH} up to a constant there exists a unique $\theta$-psh function $\varphi$ such that $T=\theta+dd^c \varphi$. The function $\varphi$ is called the \textit{global potential} of $T$. Finally, a set $K\subset X$ is called $\psh$-\textit{pluripolar} if $K\subset \{\varphi=-\infty\}$ for some qpsh function $\varphi\in\psh.$
 %Recall that \textit{Lelong number} of a positive closed (1,1) current $T$ is defined by 
 %$$ \nu(T,x):=\displaystyle \liminf_{z\rightarrow x}\frac{\phi(z)}{\log|x-z|}$$
 %where $\phi$ is a local potential for $T$ near $x$ that is $\phi$ is psh and $T=dd^c \phi$ near $x$. This definition is independent of the choice of the local potential $\phi$ and the local coordinates. 
 %Roughly speaking, Lelong number of a positive closed current $S$ at a point $x$ quantifies the singularity of a local potential of $S$ at the point x. 
 %For instance, Lelong number coincides with the algebraic multiplicity in the case of a current of integration.

\subsection{Comparison of Singularities}
Let $\varphi_1,\varphi_2$ be two $\theta$-psh functions. Following \cite{DPS01}, we say that $\varphi_1$ is less singular than $\varphi_2$ if $\varphi_2\leq \varphi_1+O(1).$ For positive closed currents $T_1,T_2 \in \alpha,$ we say that $T_1$ is less singular than $T_2$ if the global potential $T_1$ is less singular than the global potential of $T_2$ for some (equivalently for all) smooth representative. A $\theta$-psh function $\psi$ is said to be minimally singular if for any $\varphi\in \psh$ we have $\varphi \leq \psi+O(1).$ Similarly, a positive closed current $T$ is called minimally singular in its cohomology class if so is its global potential for some fixed (or equivalently any) smooth representative for $\{T\}$. It was observed in \cite{DPS01} that minimally singular currents always exist. Indeed,
$$\vmin:=sup\{\varphi\in \psh |\ \varphi\leq 0\}$$
defines a minimally singular $\theta$-psh function.
%Following \cite{Bou} we define the \textit{Ample locus}, $Amp(\alpha)$ of a class $\alpha \in \cohomology$ as to be the set of points $x \in X$ such that there exists a K\"ahler current $T\in \alpha$ which is smooth near $x$. Thus, we can choose $\theta \in \alpha$ smooth such that $T=\theta+ dd^c \vmin$ where $$\vmin=sup\{\varphi\in \psh| \varphi\leq 0\}.$$
Notice that minimally singular currents are not unique in general. However, once  we fix $\theta$ there is a canonical one, namely $\vmin.$ 
%Hence, we may define set $D_{\theta}:=\{\vmin=0\}$. Note that $D_{\theta}$ is a compact set since $\vmin$ is usc. For example, if $\theta$ is a semi-positive form then $V_{\theta}$ is identically zero everywhere and $D_{\theta}=X.$ 

 %In the sequel, for $\alpha \in \bigc$ and $\theta \in \alpha$ be a fixed smooth representative. Let $\vmin$ denote the upper envelope of the negative $\theta$-psh functions which is uniquely defined and  %\begin{question}
 %Let $\theta'$ be another smooth form in $\alpha$. What is the relationship between $D_{\theta}$ and $D_{\theta'}?$
 %\end{question} 
\subsection{Non-pluripolar Monge-Amp\`ere}
It is well known that wedge product of positive closed currents is not always well-defined. However, if $T_1,\dots,T_p$ are positive closed $(1,1)$ currents with locally bounded potentials $u_j$'s, Bedford and Taylor \cite{BT1,BT2} proved that 
$$dd^cu_1\wedge \dots \wedge dd^cu_p$$
is a well-defined positive closed bidegree $(p,p)$ current and depends only on the currents $T_j$'s but not the choice of local potentials  $u_j$'s. Moreover, the wedge product is local in the plurifine topology in the sense that if $v_j$ are psh function such that $u_j=v_j$ a.e on a plurifine open set $\mathcal{O}$ then $$1_{\mathcal{O}}dd^cu_1\wedge \dots \wedge dd^cu_p=1_{\mathcal{O}}dd^cv_1\wedge \dots \wedge dd^cv_p.$$
More recently, the authors of \cite{BEGZ} showed that for arbitrary positive closed $(1,1)$ currents $T_1,\dots,T_p$ on a compact K\"ahler manifold there is a canonical way to define a \textit{non-pluripolar product}
$$\langle T_1\wedge\dots\wedge T_p\rangle$$
which is a globally well-defined positive closed $(p,p)$ current and does not put any mass on pluripolar sets.\\ \indent
 We say that a positive closed $(1,1)$ current $T$ has \textit{small unbounded locus} if there exists a closed complete pluripolar set $A$ such that $T$ has locally bounded potentials on $X\setminus A.$ It follows from Demailly's approximation theorem \cite{Dem} that if $\alpha \in \bigc$ then there are plenty of positive closed currents in $\alpha$ with small unbounded locus. The non-pluripolar product is increasing in the following sense:
\begin{prop}\cite{BEGZ}
Let $T_j,S_j\in \alpha_j$ be positive closed $(1,1)$ currents with small unbounded locus such that $T_j$ is less singular than $S_j$. Then the cohomology classes satisfy $$\{\langle T_1\wedge\dots\wedge T_p\rangle\} \geq  \{\langle S_1\wedge\dots\wedge S_p\rangle\} $$
in the sense that the difference can be represented by a positive closed $(p,p)$ current.
\end{prop}
 We refer the reader to the papers \cite{BEGZ,BBGZ} for details and further properties of non-pluripolar products. 
\begin{defn} \label{def}
Let $\alpha_1,\dots,\alpha_p \in \bigc$ and $T^{min}_j \in \alpha_j$ be a positive closed current with minimal singularities. Then the cohomology class 
$$\langle\alpha_1\dots \alpha_p\rangle:=\{\langle T_1^{min}\wedge \dots \wedge T_p^{min}\rangle\} \in H^{p,p}_{pef}(X,\Bbb{R})$$
is independent of the choice of the currents of minimal singularities $T^{min}_j\in \alpha_j.$
\end{defn}
By using continuity of $\langle\alpha_1\dots \alpha_p\rangle$ on p-tuples of big classes and  monotonicity of non-pluripolar products one can extend this definition to merely psef classes by setting
$$\langle \alpha_1\dots \alpha_p\rangle=\lim_{\epsilon \to 0} \langle(\alpha_1+\epsilon \beta)\dots (\alpha_p+\epsilon \beta) \rangle$$
where $\beta$ is any K\"ahler class (\cite{BEGZ}). If $\alpha_1=\alpha_2\dots =\alpha_p$ we write $\langle \alpha^p\rangle := \langle \alpha \dots \alpha\rangle.$ In particular, the non-negative number 
$$vol(\alpha):= \langle\alpha^k\rangle$$
is called the \textit{volume} of $\alpha.$ Then a psef class $\alpha$ is big if and only if $vol(\alpha)>0$ \cite{B02, DP}. It follows from \cite{B02} that if $L$ is a big holomorphic line bundle and $\alpha=c_1(L)$ is the first Chern class then $vol(\alpha)$ coincides with the volume of $L$ introduced by Fujita. Namely,
$$ Vol(L)=\limsup_{m\to \infty}\frac{k!}{m^k}h^0(X,L^{\otimes m}).$$
% Following \cite{BDPP}, we define the \textit{numerical dimension} of a pseudo-effective class $\alpha$ by 
 %$$\nu(\alpha):=\max\{p: \langle \alpha^p\rangle \not=0\ \text{in} \ H^{p,p}(X,\Bbb{R})\}$$
%Note that $0\leq \nu(\alpha)\leq n$. Moreover, $\nu(\alpha)=n$ if and only if $\alpha$ is big.\\ 
 Finally, we say that a positive closed current $T\in \alpha$ has \textit{full Monge-Amp\`ere} if $$\int_X\langle T^k\rangle=vol(\alpha).$$
 Note that this is always the case if the class $\alpha$ is not big.
\subsection{Energy Clases}\label{ener}
An \textit{admissible weight} is a smooth convex increasing function $\chi:\Bbb{R} \to \Bbb{R}$ such that $\chi(-\infty)=-\infty$ and $\chi(t)=t$ for $t\geq 0.$ Following \cite{BEGZ} we define $\chi$-energy of a minimally singular $\theta$-psh function $\varphi$ by
$$E_{\chi}(\varphi):= \frac{1}{n+1}\sum_{j=0}^{n}\int_X(-\chi)(\varphi-\vmin)\langle T^{j}\wedge T_{min}^{n-j}\rangle$$
where $T=\theta+dd^c\varphi$ and $T_{min}=\theta+dd^c \vmin.$
Note that if $\chi(t)=t$ then $E_{\chi}(\varphi)$ coincides with the Aubin-Mabuchi energy functional (up to a minus sign). Let $\varphi \in \psh$ be an arbitrary $\theta$-psh function and $\varphi_k=\max(\varphi,\vmin-k)$ denote the \textit{canonical approximants} of $\varphi$ then it follows from \cite[Theorem 2.17]{BEGZ} that one can define $$E_{\chi}(\varphi):=\sup_kE_{\chi}(\varphi_k) \in (-\infty,+\infty].$$ If the right hand side is finite, we say that $\varphi$ has \textit{finite energy}.
\begin{defn}
Given an admissible weight function $\chi$, the set $\energyc$ denotes the set of all $\theta$-psh functions with finite energy. In particular, if $\chi(t)=- (-t)^p$ we set $\energyp:=\energyc.$ Moreover, we let $\energy$ denote the set of all $\theta$-psh functions with full Monge-Amp\`ere.
\end{defn}
Note that the $\energyp$ is decreasing in the sense that $\energyp\subset \mathcal{E}^q(X,\theta)$ whenever $1\leq q<p.$
\begin{prop}\cite{BEGZ}
Let $\alpha \in \bigc$ and $\theta \in \alpha$ be a smooth representative then $$\energy=\cup_{\chi}\energyc$$
where $\chi$ ranges over all admissible weight functions.
\end{prop}
Finally, we remark that the global potential of a positive closed current belongs an energy class is independent of the choice of smooth representative. Thus, we write $T \in \energypc$ if $\varphi \in \energyp$ where $T=\theta+dd^c\varphi$ for some smooth representative $\theta \in \alpha.$
\subsection{Monge-Amp\`ere Capacity}
Let $\theta$ be a smooth real closed $(1,1)$ form on $X$ such that $\{\theta\}\in \psef.$ For $\varphi \in \psh$ we let $MA(\varphi)$ denote the non-pluripolar Monge-Amp\` ere $\langle (\theta+dd^c\varphi)^k\rangle.$ Following \cite{BEGZ}  we define the pre-Capacity of a Borel subset $B\subset X$
$$ Cap_{\theta}(B):=\sup\{\int_B MA(\varphi)\ |\ \varphi\in Psh(X,\theta), \vmin-1 \leq \varphi \leq \vmin\}.$$
Note that $0\leq Cap_{\theta}(B)\leq \int_XMA(\vmin) =vol(\alpha).$\\
For a Borel set $K \subset X$ we define 
$$V_{K,\theta}:=sup\{\varphi \in \psh: \varphi\leq 0\ \text{on}\ K\}.$$
Let $V_{K,\theta}^*$ denotes the upper semi-continuous regularization of $V_{K,\theta}$. Then it follows that $V_{K,\theta}^* \in \psh$ if and only if $K$ is not $\psh$-polar \cite{GZ}. The function $V_{K,\theta}^*$ is called \textit{global extermal function} of $K$.
%\begin{prop}\label{V}
%Let $\alpha \in \bigc$ and  $\varphi \in  \energy$ then $\nu(\varphi,x)=0$ for every $x\in D_{\theta}.$
 %\end{prop}
%\begin{proof}
%Assume the contrary that $\nu(\varphi,x_0)>0$ for some $x_0\in D_{\theta}.$ Then locally $\varphi \sim c\log dist(\cdot,x_0)+\gamma$ for some constants $c,\gamma>0.$ However, the later has well defined Monge-Ampere \cite{Demailly} near $x_0$ and puts positive mass on $\{x_0\}$. Hence, $$\int_X MA(\varphi)<\int_X MA(\vmin)=vol(\alpha)$$ and this contradicts that $\varphi \in \energy.$
%\end{proof}
\section{Dynamics of Meromorphic Maps}
Recall that a \textit{meromorphic map} $f:=\pi_1^{-1}\circ \pi_2:X\to X$ of a complex manifold is defined by its graph which is an irreducible analytic subvariety $\Gamma_f \subset X\times X$ together with projections $\pi_1,\pi_2:\Gamma_f\to X$ such that the projection $\pi_1:\Gamma_f \to X$ onto the first factor is a proper modification. This means that there exists an analytic set $I_f\subset X$ of codimension at least two such that $\pi_1$ is an isomorphism from $\Gamma_f \setminus \pi_1^{-1}(I_f)$ onto its image. We denote the indeterminacy set of the $n^{th}$ iterate $f^n$ by $I_{f^n}$ and set $\mathcal{I_1}:=\cup_{n=1}^{\infty} I_{f^n}.$ We say that $f$ is \textit{dominant} if the projection $\pi_2:\Gamma_f\to X$ is surjective. \\ \indent
For a real smooth closed $(p,p)$ form $\theta$ on $X$ we define its \textit{pull-back} by
$$f^*\theta:=(\pi_1)_*(\pi_2)^*\theta$$ where the later push-forward considered as a current. This definition induces a linear action on the cohomology $$f^*:\cohomologyp \to \cohomologyp$$
$$f^*\{\theta\}=\{(\pi_1)_*(\pi_2^*)\theta\}$$
where $\{(\pi_1)_*(\pi_2^*)\theta\}$ denotes the de Rham class of the closed current $(\pi_1)_*(\pi_2^*)\theta.$
%In particular, if $f:\pk\to \pk$ is a rational map then $f^*$ may be identified with the algebraic degree of $f.$ 
In general, this linear action is not compatible with the dynamics of the map. We say that $f$ is \textit{p-regular} if $(f^n)^*=(f^*)^n$ on $\cohomologyp$ for each $n=1,2\dots$ %For example, if $f$ is a rational map of $\pk$ then $f$ is p-regular if and only if $\deg(f^n)=\deg(f)^n$ where $f^n$ denotes the $n^{th}$ iterate of $f$. 
\\ \indent
For $0\leq p\leq k$ we let $$\delta_p(f^n):=\int_{X\setminus I_{f^n}} (f^n)^*\omega^p\wedge\omega^{k-p}.$$
Then the $p^{th}$ \textit{dynamical degree} is defined by $$\lambda_p(f):=\liminf_{n \to \infty}(\delta_p(f^n))^{\frac{1}{n}}$$
where $\omega$ is an auxiliary K\"ahler form on $X$. Note that $d_{top}:=\lambda_k$ is the topological degree of $f$ which is the number of pre-images of a generic point. %Roughly speaking, the dynamical degrees measure the asymptotic growth rate of volumes of iterated subvarieties. 
If $f$ is p-regular then $\lambda_p$ coincides with spectral radius of $f^*.$ %Therefore, the role of algebraic degree in the case of rational mappings of $\pk$ is played by $\lambda_1$ in meromorphic setting. 
Moreover, the dynamical degrees are invariant under bimeromorphic conjugations \cite{G,DS04}.\\ \indent
Recall that one can define the pull-back map on the set of positive closed $(1,1)$ currents $$T\to f^*T$$ by pulling back the local potentials. It follows that this linear action is continuous in the weak topology of positive currents \cite{Meo}. Moreover $f^*$ preserves classes i.e. $f^*\{T\}=\{f^*T\}$ where $\{T\}\in \cohomology$ denotes the de Rham cohomology class of closed current $T$. Therefore, $f^*\psef \subset \psef.$ Since $\psef$ is a closed salient convex cone it follows from a Perrron-Frobenius type argument \cite{DF} that there exists $\alpha\in\psef$ such that $f^*\alpha=\lambda_1(f)\alpha.$

\subsection{Volume Estimates}
 In this section we will study the rate of contraction of volume of a set under the iteration by a dominant meromorphic map. The following result is a slightly improved version of  \cite[Theorem 0.1]{G04} (see also \cite{DS08}) and allows us to derive certain equidistribution results. 
 \begin{thm} \label{vol}
 Let $f:X\dashrightarrow X$ be a dominant meromorphic map. We assume that $\lambda:=\lambda_1(f)>1$ and $\delta_1(f^n)\leq C \lambda^n$ for some constant $C>0$ and for every $n \in \mathbb{N}$. Then for any Borel set $\Omega \subset X$ and $n\in \mathbb{N}$
 $$Vol(f^n(\Omega))\geq (C_1Vol(\Omega))^{C_2\lambda^n}$$
 where $C_1, C_2>0$ independent of $n$ and  $\Omega$.
 \end{thm}
 \begin{proof}
 We will utilize some ideas from \cite{Di,FJ,G04}. We define the measure $f^*\omega^k$ to be the trivial extension of $(f|_{X\setminus I_f})^*\omega \wedge \dots \wedge (f|_{X\setminus I_f})^*\omega$ through $I_f$ and let $Jac_{\omega}(f)$ denote the complex Jacobian of $f$ with respect to $\omega^k$ that is defined by
 $$f^*\omega^k=|\jac|^2 \omega^k.$$
 Notice that $|\jac|^2$ is a nonnegative continuous function on $X\setminus I_f.$
 First, we claim that $\log|\jac|=u_1-u_2$ where $u_i$ are qpsh functions. Indeed, in a local coordinate chart we may write $\omega=dd^c \psi$ for some smooth strictly psh function $\psi.$ Then
 $$\omega^k=(dd^c\psi)^k=4^kk! \det[\frac{\del^2\psi}{\del z_i \del \bar{z}_j}] dA$$
 where $dA=(\frac{i}{2})^kdz_1\wedge d\bar{z}_1\wedge \dots \wedge dz_k\wedge d\bar{z}_k$ and $z=(z_1,\dots,z_k)$ indicates the local coordinates. This implies that 
 \begin{equation}\label{local}
 |\jac|^2=\fraction{\det[\frac{\del^2\psi}{\del z_i \del \bar{z}_j}]  \circ f}{\det[\frac{\del^2\psi}{\del z_i \del \bar{z}_j}] } |Jac_{local}(f)|^2\end{equation}
where $Jac_{local}(f)$ denotes the Jacobian of $f$ in the local coordinates. Thus, the claim follows by covering $X$ with coordinate charts and observing that in the intersection of such sets the RHS of (\ref{local}) is well-defined up to multiplication by a non-vanishing holomorphic function.
  
\begin{lem}\label{dsh}
If $\delta_1(f^j)\leq C\lambda^j$ then there exists $c>0$ and  qpsh functions $u_{1,j},u_{2,j}$ such that $$\log |Jac_{\omega}(f^j)|=u_{1,j}-u_{2,j}\ \text{and}\ 
dd^cu_{i,j}\geq -c\lambda^j\omega$$ for $i=1,2$ and $j\geq1.$
\end{lem}
\begin{proof}
Let $\log|Jac_{\omega}(f)|=u_1-u_2$ be as above. Without lost of generality we may assume that $dd^cu_i\geq -\omega$ for $i=1,2.$
Let $\omega_1,\dots,\omega_N$ be K\"ahler forms whose cohomology classes form a basis for $\cohomology$ such that $\omega_l \leq \omega$  for $l=1,\dots, N$. For a class $\alpha \in \cohomology$ we let $\theta(\alpha) \in \alpha$ denote a smooth closed (1,1) form which lies in the linear span of $\omega_1,\dots,\omega_n$. We define a norm on $\cohomology$ by $||\alpha||:= \max_{l=1,\dots, N} |a_l|$ where $\theta(\alpha)=\sum_{l=1}^Na_l\omega_l.$ Note that $$\alpha \to \int_X \theta(\alpha)\wedge \omega^{k-1}$$ defines a continuous linear functional on $\cohomology$ and it is non-negative on $\psef$ where the later is a closed convex cone. Then there exists $C_1>0$ such that
$$||\alpha||\leq C_1 \int_X \theta(\alpha)\wedge \omega^{k-1}$$
for every $\alpha \in \psef.$
Letting $(f^j)^*\omega=\theta_j+dd^c \phi_j$ where $\theta_j:=\theta(\{(f^j)^*\omega\})$ we infer that $\theta_j\leq C_2\delta_1(f^j)\omega$ for some $C_2>0$. By the chain rule we have $$Jac_{\omega}(f^j)=\prod_{l=0}^{j-1}Jac_{\omega}(f)\circ f^l.$$
Then letting 
\begin{equation}\label{jaceq}
u_{i,j}=\sum_{l=o}^{j-1}(\phi_l+u_i\circ f^l)
\end{equation}
 and using hypothesis we conclude that 
$$dd^cu_{i,j}\geq -C_3 \sum_{l=0}^{j-1}\lambda^l \omega\geq -c\lambda^j\omega$$ for some $c>0$ independent of $j.$
\end{proof}
Next, we prove the following improved version of Lemma \ref{dsh}:
\begin{lem}
If $\delta_1(f^j)\leq C\lambda^j$ then there exists $c>0$ and qpsh functions $u^{\pm}_j$ such that
$$\log|Jac_{\omega}|=u_j^+-u_j^-+ c_j$$
where $\sup_Xu_j^{\pm}=0,\ dd^cu_j^{\pm}\geq -c\lambda^j\omega$ and $|c_j|\leq c\lambda^j.$
\end{lem}
\begin{proof}
 By Lemma \ref{dsh} we may write
 $$\log|Jac_{\omega}|=u_j^+-u_j^-+ c_j$$
 with $\sup_Xu_j^{\pm}=0$ and $dd^cu_j^{\pm}\geq -c\lambda^j\omega$ where $u_j^{\pm}:=u_{i,j}-\sup_Xu_{i,j}.$ We need to estimate $c_j.$ Note that $\mathcal{F}=\{\varphi \in Psh(X,\omega): \sup_X\varphi=0\}$ is a compact family in $L^1(X),$ hence there exists a constant $C_1>0$ such that
 $$-C_1\leq \int_X u_j^{\pm}\leq 0$$
 for every $j\geq 1.$  
 By (\ref{jaceq}) and \cite[Proposition 1.3]{G04} we have $$\|\log|Jac_{\omega}(f^j)|\|_{L^1(X)}=\|\sum_{l=0}^{j-1}(u_1-u_2)\circ f^l\|_{L^1(X)}\leq C\lambda^j$$ thus,
 $$|c_j|\leq \|\log|Jac_{\omega}|\|_{L^1}+\|u_j^+\|_{L^1}+\|u_j^-\|_{L^1}\leq c\lambda^j$$
for some $c>0$ independent of $j.$
\end{proof}
We continue with the proof of Theorem \ref{vol}. Let us denote $\nu:=2\sup\nu(\varphi,x)>0$
where the supremum is taken over all $x\in X$ and $\varphi\in Psh(X,\omega).$ Since $X$ is compact $\nu$ is finite and depends only on the class $\{\omega\}$. Then by uniform version of Skoda's integrability theorem \cite{Ze} applied to the compact family $\mathcal{F}_{\nu}=\{\varphi \in Psh(X,\frac{1}{\nu}\omega): \sup_X\varphi=0\}$ we obtain
$$\int_X \exp(-\frac{1}{\nu C \lambda^j}(u_j^+-u_j^-+c_j))dV\leq C_1\int_X \exp(-\frac{1}{\nu C \lambda^j}u_j^+)dV\leq C_2<\infty.$$
where $C_1,C_2>0$ does not depend on $j.$ Then for $t>0$ by Chebyshev's inequality
\begin{equation}\label{estimate}
Vol(u_j^+-u_j^-+c_j<-t)\leq C_2\exp(-\frac{1}{C \lambda^j}t). 
\end{equation}
Now, we fix $t>0$ such that $C_2 t^{\frac{1}{C \lambda^j}}=\frac{1}{2}Vol(\Omega)$ then by applying change of variables, Chebyshev's inequality and (\ref{estimate}) we obtain 
\begin{eqnarray*}
Vol(f^j(\Omega)) & \geq & \frac{1}{d_{top}}\int_{\Omega}|Jac_{\omega}(f^j)|^2dV \\ 
& \geq & \frac{t}{d_{top}} Vol(\{|Jac_{\omega}(f^j)|^2\geq t\}\cap \Omega)\\ 
&= & \frac{t}{d_{top}}(Vol(\Omega)-Vol(\{|Jac_{\omega}|^2<t\}) \\
&\geq & \frac{t}{d_{top}}(Vol(\Omega)-C_2t^{\frac{1}{C \lambda^j}}) \\
&\geq & \frac{t}{2d_{top}}Vol(\Omega) \\
& \geq & (C_3Vol(\Omega))^{C_4\lambda^j} 
\end{eqnarray*}
where $C_3,C_4>0$ do not depend on $j.$
 \end{proof}
We stress that the volume estimates does not require $f$ to be 1-regular. Moreover, these estimates are sharp as already observed in \cite{G04}.\\ \indent
 A direct proof of the next corollary can be found in \cite{G04}. However, we give a proof for the convenience of the reader. 
 \begin{cor} \label{UI}
 Let $f$ be as in Theorem \ref{vol} and $\phi$ be a qpsh function on $X$. Then the sequence $\{\frac{1}{\lambda_1(f)^n}\phi \circ f^n\}$ is uniformly integrable.
 \end{cor}
 \begin{proof}
 By scaling and translating $\phi$ if necessary we may assume that $dd^c\phi\geq -\omega$ and $\phi \leq 0.$ Let $\lambda:=\lambda_1(f).$  For $\alpha>0$ and $n$ is fixed we set $E^n_{\alpha}:=\{\frac{1}{\lambda^n}\phi\circ f^n<-\alpha\}$ then
 \begin{eqnarray*}
 \displaystyle
 \int_{E^n_{\alpha}}-\frac{1}{\lambda^n}\phi \circ f^n dV= \alpha\ Vol(E^n_{\alpha}) + \int_{\alpha}^\infty Vol(E^n_t) dt
 \end{eqnarray*}
 By \cite[Theorem 3.1]{Ki}there exists constants $A,B>0$, independent of $n$, such that $$Vol(\phi<-t \lambda^n)\leq A\exp(-Bt\lambda^n).$$ Since
 $$f^n(\{\frac{1}{\lambda^n}\phi\circ f^n<-t\}) \subset \{\phi<-t \lambda^n\}$$ we infer from Theorem \ref{vol}  that 
 \begin{eqnarray*}
 \displaystyle
 \int_{E^n_{\alpha}}-\frac{1}{\lambda^n}\phi \circ f^n dV  \leq  \exp(-C_1 \alpha) (\alpha+ C_2).
\end{eqnarray*}
for some constants $C_1,C_2>0$ independent of $n$.
 \end{proof}

 \section{Green Currents}
Green currents play an impotant role in the dynamical study of meromorphic maps %Construction of Green currents has a long history in complex dynamics 
(see \cite{Si,G04,DDG} and references therein). In the case of bidegree $(1,1)$, we have the following existance theorem:%one of the main results in \cite{B} provides a general construction of an $f^*$-invariant \textit{Green current} which represents an expanding direction for $f^*$ in the cohomology and has good convergence properties:
 \begin{thm}\cite{B} \label{B}
Let $f:X\dashrightarrow X$ be a 1-regular dominant meromorphic map and $\alpha \in H^{1,1}_{psef}(X,\mathbb{R})$ such that $f^*\alpha =\lambda \alpha$ for some $\lambda > 1$. If 
\begin{equation*} 
\frac{1}{\lambda^{n}} \vmin \circ  f^{n} \rightarrow  0 \ \text{in} \ L^1(X) \tag{$\star$}
\end{equation*}
for some (equivalently for every) smooth representative $\theta$ of $\alpha$ then  there exists a positive closed $(1,1)$ current $T_{\alpha}\in \alpha$ such that for every smooth form $\eta \in \alpha$ $$\lim_{n\rightarrow \infty}\frac{1}{\lambda^n}(f^{n})^*\eta \to T_{\alpha}$$ in the sense of currents.
Furthermore, 
\begin{itemize}
\item[(1)] $T_{\alpha}$ is minimally singular among the invariant currents which belong to the class $\alpha$.
\item[(2)] $T_{\alpha}$ is extreme within the cone of positive closed (1,1) currents whose cohomology class belongs to $\mathbb{R}^+\alpha$.
\end{itemize}
 \end{thm} 
The dynamical assumption $(\star)$ holds in a quite general setting. For instance, if the class $\alpha$ can be represented by a smooth semi-positive form then $\vmin$ is identically zero on $X$ and hence $(\star)$ holds. This is always the case when $X$ is a complex homogeneous manifold.
\subsection{Normal Maps}
\indent We say that a point $x \in X$ is \textit{normal} if there exists disjoint open sets $x \in U$ and $\overline{\mathcal{I_1}} \subset V$ such that $f^n(U)\cap V=\emptyset$ for every $n\in \mathbb{N}.$ We denote the set of normal points by $\mathcal{N}_f.$ We say that $f$ is normal if $\mathcal{N}_f=X\setminus\overline{\mathcal{I_1}}.$ Complex H\'enon maps or more generally regular polynomial automorphisms of $\Bbb{C}^k$ are among the exmples of normal mappings (see \cite{BS,Si}). 
 \begin{prop} \label{normal}
 Let $f:X\to X$ be a dominant 1-regular meromorhic map. Assume that $\lambda:=\lambda_1(f)>1$ is a simple eigenvalue of $f^*|_{H^{1,1}(X,\mathbb{R})}$ and $f^*\alpha=\lambda \alpha$ such that $\alpha \in \bigc$. If $f$ is normal and $Vol( \overline{\mathcal{I_1}(f)})=~0$ then $(\star)$ holds.
 \end{prop}
 \begin{proof}
 Let $\theta$ be a fixed smooth representative for $\alpha$ and $T_{\theta}:= \theta+ dd^c \vmin$ denote the minimally singular current. It follows from \cite[Theorem 3.3] {B} that 
 $$ E_+(\Tmin)\subset  \mathcal{I_1}(f).$$
 Let $$ \Omega_{n, t}=\{x\in X: \fraction{1}{\lambda^{n}} \vmin \circ f^n(x) <-t \}. $$
 Since $Vol(\overline{\mathcal{I_1}(f)})=0$ by Corollary \ref{UI} it is enough to prove that $Vol(U \cap \Omega_{n,t})\rightarrow 0$ as $n \rightarrow \infty$ for every relatively compact subset $U \subset X\setminus \overline{\mathcal{I_1}(f)}.$ 
 For given such $U$ by normality of $f$, there exists an open subset $V \subset X$ such that $\overline{\mathcal{I_1}(f)} \subset V$ and $f^n(U) \cap V=\emptyset$ for every $n \in \mathbb{N}.$ Notice that,
 $$f^n(U \cap \Omega_{n,t}) \subset \{x\in X\setminus V: \vmin<-t \lambda^n\}.$$  
 Since $\nu(\Tmin,x)=0$ for every $x\in X\setminus V$, it follows from Theorem \ref{vol} and \cite[Theorem 3.1]{Ki} that for every $A>0$ there exists $B>0$ such that  
$$(C_1Vol(U\cap\Omega_{n,t}))^{C_2\lambda^n}\leq B \exp(-At\lambda^n)$$
where $C_1,C_2>0$ as in Theorem \ref{vol}. Since $A>0$ can be chosen arbitrarily large the assertion follows. 
 \end{proof}
We remark that the assumption $Vol( \overline{\mathcal{I_1}(f)})=0$ in Proposition \ref{normal} is a non-trivial condition (cf. \cite[\S 6.1]{DG}).
\section{Equidistribution towards Green Current}
%In this section we will consider the following problem: Let $(X,f)$ be as in Theorem \ref{B}. Provide sufficient conditions on the set of positive closed $(1,1)$ currents $S$ on $X$ such that $$\frac{1}{\lambda^n}(f^n)^*S \to T_{\alpha}$$
%in the sense of currents. %In this case, we say that the positive closed current $S$ \textit{equidistributes towards Green current} $T_{\alpha}.$
 \subsection{Equidistribution of Zeros of Sections } \label{random}
Recall that a compact K\"ahler manifold $X$ carries a big line bundle if and only if $X$ is projective. Let $X$ be a projective manifold and $\pi:L\to X$ be a hermitian holomorphic line bundle. We denote the semi-group
$$N(L):=\{m\in \Bbb{N}: H^0(X,L^{\otimes m})\not=0\}.$$ Given $m\in N(L),$ we consider the \textit{canonical map} induced by the complete linear series $|L^{\otimes m}|$  $$\Phi_m:X \dashrightarrow \Bbb{P}^{N_m}$$
$$x \to [s_0(x):s_1(x):\dots :s_{N_m}(x)]$$  
where the identification $\Bbb{P}H^0(X,L^{\otimes m})^{\star}=\Bbb{P}^{N_m}$ is determined by  the choice of the basis $s_0,\dots,s_{N_m}$ for $H^0(X,L^{\otimes m}).$ Note that $\Phi_m$ is a rational map which is holomorphic on the complement of the base locus $B_{|L^{\otimes m}|}:=\bigcap_{s\in H^0(X,L^{\otimes m})}s^{-1}(0).$ We also let $Y_m:=\Phi_m(X)\subset \Bbb{P}^{N_m}$ be the image of the closure of graph of $\Phi_m$. Recall that a line bundle $L$ is called \textit{semi-ample} if $B_{|L^{\otimes m}|}=\emptyset$ for some $m>0.$ The \textit{Kodaira-Iitaka dimension} of $L$ is defined to be $$\kappa(X,L):=\max_{m\in N(L)}\dim Y_m$$
if $N(L)\not=\emptyset$ otherwise we set $\kappa(X,L)=-\infty.$ \textit{Kodaira dimension} of $X,$ denoted by $kod(X),$ is defined to be $\kappa(X,K_X)$ where $K_X$ is the canonical bundle of $X$. Notice that kod(X) is equal to one of $-\infty,0,1,\dots, k.$  
 %Recall that a holomorphic line bundle $L$ is called \textit{semi-ample} if $L^{\otimes m}$ is globally generated by sections for some $m>0.$ Then for any such $m$ the linear system $|L^{\otimes m}|$ induces a holomorphic map $\Phi_m:X\to \Bbb{P}^{N_m}.$ 
 A line bundle $L$ is called \textit{big} if $\kappa(X,L)=k.$ We stress that this definition is consistent with the one in Section \ref{prelim} that is $L$ is big if and only if $c_1(L)$ contains a K\"ahler current \cite{B02}. %It is a consequence of Iitaka fibration theorem that a line bundle $L$ is big if and only if the mapping $\Phi_m$ is birational onto its image for some $m>0$ (\cite{L}). 
It is well-known that $L$ is big if and only if there exists $C>0$ such that $h^0(X,L^{\otimes m})\geq Cm^k$ for sufficiently large $m>0$ (see \cite{L}). \\ \indent %By Kodaira embedding theorem, $X$ is projective if and only if it admits an ample line bundle. Moreover, $X$ is Moishezon (i.e bimeromorphic to a projective manifold) if and only if the intersection of $\bigc$ with $NS_{\Bbb{R}}(X)$ is non-empty \cite{DP}.\\ \indent 
Let $h$ be a fixed smooth metric on $L$ and $\theta_h$ denote its curvature form which is a globally well-defined smooth $(1,1)$ form representing the first Chern class $c_1(L)\in H^2(X,\Bbb{Z})$ which we consider as an element of $\cohomology$ by identifying it with its image under the mapping $i:H^2(X,\Bbb{Z})\to \cohomology$ induced by $i:\Bbb{Z}\to \Bbb{R}.$ 
  The metric $h$  induces a metric on $L^{\otimes m}$ which we denote  by $h_m$. %For $s=\{s_{\alpha}\} \in H^0(X,L^{\otimes m})$ we define its $h_m$-norm by $|| s||_{h_m}:=|s_{\alpha}|e^{-m\psi_{\alpha}}$ on $U_{\alpha}.$ Note that this definition is independent of $\alpha$ due to the compatibility conditions $s_{\alpha}=s_{\beta}.g_{\alpha\beta}$ on $U_{\alpha\beta}$. 
 We also define an $L^2$-norm on $H^0(X,L^{\otimes m})$ by
  $$||s||^2_{L^2(dV,L^{\otimes m})}:=\int_X||s||^2_{h_m} dV$$
and  we  write $|s|$ for $||s||_{L^2(dV,L^{\otimes m})}$ for short. \\ \indent
  % Moreover, $L^{\otimes m}=\Phi_m^*(\mathcal{O}_{\Bbb{P}^N}(1))$ and $H^0(X,L^{\otimes m})=\Phi_m^*(H^0(\Bbb{P}^N\mathcal{O}_{\Bbb{P}^N}(1)).$ %We also let $SH^0(X,L^m)$ denote the unit sphere in $H^0(X,L^m)$ relative to this inner product and $d\sigma_m$ denotes the Haar measure on $SH^0(X,L^m)$.
%Then by Poincar\'e-Lelong formula locally we can write $$[Z_s]=dd^c\log|s_{\alpha}|$$ hence, $$[Z_s]=m\theta_h+dd^c\log||s||_{h_m}$$ where the equality follows from compatibility conditions. Thus, we conclude 
For $s\in H^0(X,L^{\otimes m})$ we let $[Z_s]$ denote the current of integration along the zero divisor of $s.$ 
Note that $\frac{1}{m}[Z_s]$ is a positive closed $(1,1)$ current representing the class  $c_1(L).$ \\ \indent
Now, we define $$SH^0(X,L^{\otimes m}):=\{s \in H^0(X,L^{\otimes m}): |s|=1\}$$
to be the unit sphere in the Hilbert space $H^0(X,L^{\otimes m}).$ We identify $SH^0(X,L^{\otimes m})$ with the $(2N_m-1)$ sphere in $\Bbb{C}^{N_m}$ and regard it as a probability space endowed with image the unitary invariant measure denoted by $\mu_m$ (or equivalently we may consider $\Bbb{P}H^0(X,L^{\otimes m})$ furnished with the Fubini-Study volume form). 

% \begin{thm}\label{finite}
 %Let $f:X\to X$ be a 1-regular dominant rational map and $\lambda:=\lambda_1(f)>1$ be simple eigenvalue of $f^*$ with $f^*\alpha=\lambda\alpha$ where $\alpha=c_1(L)$ for some big line bundle $L\to X$. Assume that $L\cdot C \geq 0$ for every irreducible algebraic curve $C\subset \pi_2\circ\pi_1^{-1}(I_f).$ If the algebra $$R(L)=\bigoplus_{m\geq0}H^0(X,L^{\otimes m})$$ is finitely generated then there exists $m>0$ and a pluripolar set $\mathcal{E}_m\subset |L^{\otimes m}|$ such that for every $H\not\in \mathcal{E}_m$ 
 %and any minimally singular current $T_{min} \in c_1(L)$
  %$$\lambda^{-n}(f^n)^*(\frac{1}{m}[H]) \to T_{\alpha}$$
  %in the sense of currents as $n\to \infty.$ 
 %\end{thm}  
The following proposition is well-known by algebraic geometers \cite{BouP}. However, we were unable to find the statement and the proof in the literature.
\begin{prop}\label{minimal}
The algebra $$R(L)=\bigoplus_{m\geq0}H^0(X,L^{\otimes m})$$ is finitely generated if and only if $$T_m:=\int_{SH^0(X,L^{\otimes m})}\frac{1}{m}[Z_s] d\mu_m(s)$$ is a minimally singular current in $c_1(L)$ for some $m>0.$
\end{prop}
\begin{proof}
We use the additive notation for tensor powers of $L$, we fix an orthonormal basis $\{\sigma^{m}_j\}$ for $H^0(X,mL)$.  By  \cite[Lemma 3.1]{SZ}  we have $T_m=\frac{1}{m}\Phi_m^*\omega_{FS}$. Thus,
$$T_m=\Theta_h+\frac{1}{2m}dd^c \rho_m$$
where $\rho_m:=\frac{1}{2m} \log\sum_j ||\sigma_j^m||^2_{h_m}.$ Let $h_{min}$ be a minimally singular metric on $L$ with curvature current $T_{min}=\Theta_h+ dd^c\varphi_{min}$ which is a minimally singular current in $c_1(L).$ %By scaling the functions $\rho_m$ and $\varphi_{min}$ we may assume that
%$$\rho_m\leq \rho_{m+1}\leq \varphi_{min}.$$
 %We will also need the following lemma
By \cite[Lemma 6.6]{BEGZ} the algebra $R(L)$ is finitely generated if and only if there exists $m_0>0$ such that $\rho_{km_0}=\rho_{m_0}+O(1)$ for $k \in \Bbb{N}.$
In the the rest of the proof we construct global sections of large powers of $L$ with uniform $L^2$-estimates by using an Oshawa-Takegoshi type extension theorem. We will adapt some arguments from \cite{Bou}. We fix a hermitian line bundle $(A,h_A)$ with $h_A$ has sufficiently positive curvature form $\omega_A.$ We also fix a K\"ahler current $T_+\in c_1(L)$ with global potential $\varphi.$ Then the line bundle $G_m:=mL-A $ is pseudo-effective for large $m>0$. Indeed, choosing $m_0>0$ large enough such that $m_0T_+-\omega_A$ is a K\"ahler current and writing $$G_m=(m-m_0)L+(m_0L-A)$$
we see that $T_m:=(m-m_0)T_{min}+m_0T_+-\omega_A$ is a positive current representing $c_1(G_m)$ for $m\geq m_0.$ Then we can choose a smooth Hermitian metric $h_m$ on $G_m$ such that $T_m$ is the curvature current of $\exp(-2(m-m_0)\varphi_{min}-2m_0\varphi)h_m.$\\ \indent Now, applying Oshawa-Takegoshi-Manivel $L^2$-extension theorem \cite{Man} we see that the evaluation map
$$H^0(X,G_m+A)\to \mathcal{O}_x\otimes \mathcal{J}(T_m)_x$$ 
is surjective for every $x\in X$ with an $L^2$ estimate independent of $(G_m,h_m)$ and $x.$ This means that for every $x$ and $m$ there exists $\sigma_{m,x}\in H^0(X,mL)$ such that 
$$\int_X\|\sigma_{m,x}\|^2_{h_{min}^{m-m_0}\otimes h_{T_+}\otimes h_m}dV\leq C\ \ \ \text{and} \ \ \ \|\sigma_{m,x}(x)\|_{h_{min}^{m-m_0}\otimes h_{T_+}\otimes h_m}=1$$
where $C$ is independent of $m$ and $x.$ This implies that 
\begin{eqnarray*}
\varphi_{min}(x)+\frac{1}{m-m_0}\varphi(x) & = & \frac{1}{2(m-m_0)}\log ||\sigma_{m,x}(x)||^2_{h_m} \\
& \leq & \frac{1}{2(m-m_0)}\log \sum_j ||\sigma_j^m(x)||_{h_m}^2+C \\
& = & \frac{m}{m-m_0} \rho_m(x)+C
\end{eqnarray*}
where the second line follows from the uniform bound on the $L^2$-norms. Thus, the assertion follows from $\rho_{km_0}=\rho_{m_0}+O(1)$ by taking $m=km_0$ and letting $k\to \infty.$
\end{proof}

  \begin{thm}\label{finite}
 Let $f:X\to X$ be a 1-regular dominant rational map and $\lambda:=\lambda_1(f)>1$ be simple eigenvalue of $f^*$ with $f^*\alpha=\lambda\alpha$ where $\alpha=c_1(L)$ for some big line bundle $L\to X$. Assume that $L\cdot C \geq 0$ for every irreducible algebraic curve $C\subset \pi_2\circ\pi_1^{-1}(I_f).$ If the algebra $$R(L)=\bigoplus_{m\geq0}H^0(X,L^{\otimes m})$$ is finitely generated then there exists $m>0$ and a pluripolar set $\mathcal{E}_m\subset \Bbb{P}H^0(X,L^{\otimes m})$ such that for every $s\not\in \mathcal{E}_m$ 
 
  $$\lambda^{-n}(f^n)^*(\frac{1}{m}[Z_s]) \to T_{\alpha}$$
  in the sense of currents as $n\to \infty.$ 
 \end{thm}  
If the conclusion of Theorem \ref{finite} holds for some positive integer $m$, then it also holds for any multiple of $m.$ %The positivity assumption in Theorem \ref{finite} ensures the existence of the Green current $T_{\alpha}$ (\cite{B}). 
In the special case where $X=\Bbb{P}^k$ and $L=\mathcal{O}(1)$ is the hyperplane bundle, we recover the corresponding result of \cite{RS}.
 \begin{proof}
We fix $m>0$ as in Proposition \ref{minimal} large enough such that $h^0(X,L^{\otimes m})\geq Cm^k$ for some $C>0$ and identify $H^0(X,L^{\otimes m})$ with $\Bbb{C}^{N_m}.$ %and we consider the canonical map  $\Phi_m:X\dashrightarrow \Bbb{P}^{N_m}$. 
%For $s_1,s_2 \in H^0(X,L^{\otimes m}),$ $Z_{s_1}=Z_{s_2}$ if and only if $s_1=\tau s_2$ for some $\tau\in \Bbb{C}^*.$ Thus, we can identify $|L^{\otimes m}|$ with $\Bbb{P}H^0(X,L^{\otimes m}).$ Therefore, it is enough to 
First we prove that $$\lambda_f^{-n}(f^n)^*(\frac{1}{m}[Z_s]) \to T_{\alpha}$$ for a.e. $s\in \Bbb{P}H^0(X,L^{\otimes m})$ with respect to the Fubini-Study volume form $V_m$. \\ \indent
Let $\theta$ be a smooth form in $c_1(L)$ then by \cite[Theorem 1.3]{B} 
$$\lambda^{-n}V_{\theta}\circ f^n \to 0$$in $L^1(X)$ as $n\to \infty.$ Since 
$$
T_m %& = & \int_{SH^0(X,L^{\otimes m})}\frac{1}{m}[Z_s] d\mu_m(s)\\
=  \int_{\Bbb{P}H^0(X,L^{\otimes m})}\frac{1}{m}[Z_s] dV_m(s)
$$
is a minimally singular current representing the class $\alpha$ we infer that  $$\lambda^{-n}(f^n)^*T_m\to T_{\alpha}$$ in the sense of currents.\\ \indent Now, for every $s\in \Bbb{P}H^0(X,L^{\otimes m})$ by $dd^c$-lemma \cite[pp 149]{GH} we have $$[Z_s]=mT_m+dd_x^c\varphi_s$$ where $\varphi_s \in L^1(X)$ (not necessarily qpsh!) unique up to a constant. Moreover, since $T_m$ is minimally singular $\varphi_s$ is bounded from above and we may assume that $\varphi_s\leq 0.$  To prove the assertion it is enough to show that $\lambda^{-n}\varphi_s\circ f^n \to 0$ in $L^1(X)$ for  every $s$ outside of a pluripolar set. By definition of $T_m$ we have  $$dd_x^c\int_{\Bbb{P}H^0(X,L^{\otimes m})}\varphi_s(x) dV_m(s) =0$$ and since $X$ is compact we deduce that  $$\int_{\Bbb{P}H^0(X,L^{\otimes m})}\varphi_s(x) dV_m(s) =c$$ for every $x\in X$ where $c\leq0$ is a constant. We define $$\psi_k:\Bbb{P}H^0(X,L^{\otimes m}) \to \Bbb{R}$$
 $$\psi_k(s):=\sum^k_{n=0}||\lambda^{-n}\varphi_s\circ f^n||_{L^1(X)}$$
 and claim that $\psi_k\to \psi \in L^1(\Bbb{P}H^0(X,L^{\otimes m})).$ Since $\psi_k$ is non-negative and increasing it is enough to show that $\psi_k$ is bounded in $ L^1(\Bbb{P}H^0(X,L^{\otimes m}))$ and this follows from Fubini's theorem:
 \begin{eqnarray*}
 ||\psi_k|| & = & \sum_{n=0}^k \int_X\int_{\Bbb{P}H^0(X,L^{\otimes m})}\lambda^{-n} |\varphi_s \circ f^n(x)|dV_m(s) dV(x) \\ 
 & = &|c| \sum_{n=0}^k \lambda^{-n} \leq \frac{|c|\lambda}{\lambda-1}.
 \end{eqnarray*} 
Since $\psi \in  L^1(\Bbb{P}H^0(X,L^{\otimes m})),$ $\psi(s)$ is finite for almost every $s\in \Bbb{P}H^0(X,L^{\otimes m}).$ Hence, $\lambda^{-n}\varphi_s\circ f^n \to 0$ in $L^1(X)$ for almost every $s.$  \\ \indent
  It remains to show that the convergence holds outside of a pluripolar set. Recall that by \cite[Theorem 7.2]{GZ} every locally pluripolar set $K$ is globally pluripolar that is $K\subset \{\phi=-\infty\}$ for some (globally defined) qpsh function $\phi$ on $\Bbb{P}H^0(X,L^{\otimes m}).$ Let $\mathcal{E}_m\subset \Bbb{P}H^0(X,L^{\otimes m})$ denotes the exceptional set on which $\lambda^{-n}(f^n)^*[Z_s]\not \to mT_{\alpha}.$ If $\mathcal{E}_m$ is not a pluripolar set then %it has a non-puripolar compact subset $K\subset \mathcal{E}_m$. We may asume that 
  there exists a non-pluripolar compact set $K\subset \mathcal{E}_m$ such that $K\subset \Omega$ where $\Omega\subset \Bbb{P}H^0(X,L^{\otimes m})$ is an open set contained in a coordinate chart. Then there exists a bounded psh function $u_K$ on $\Omega$  such that its Monge-Amp\`ere measure $\nu:=dd^cu_K\wedge\dots\wedge dd^cu_K$ defines a probability measure supported on $K.$ Since for every $x\in X\backslash B_{|L^{\otimes m}|}$ $$s\to \varphi_s(x)$$ is psh on $\Omega$ (cf. \cite[\S 3]{SZ}) by Chern-Levine-Nirenberg inequality we obtain
  $$\int_{K}|\varphi_s(x)|d\nu(s)\leq C\|\varphi_s(x)\|_{L^1(K)}\|u_K\|_{L^{\infty}(K)}<\infty.$$
 Thus, applying the above argument with $\nu$ in the place of $V_m$ one can conclude that $\lambda^{-n}\varphi_s \circ f^n\to 0$ for $\nu$-almost every $s\in K$. This is a contradiction. Hence, $\mathcal{E}_m$ is pluripolar.  
\end{proof}
The main idea of the proof of Theorem \ref{finite} goes back to Russakovskii-Shiffman \cite{RS} (see also \cite{Si}). Recall that for a big line bundle $L$ the algebra $R(L)$ is not always finitely generated. In fact, if $L$ is big and nef then $R(L)$ is finitely generated if and only if $L$ is semi-ample (see \cite{L} for details). In Theorem \ref{finite},  if $L$ is semi-ample then $c_1(L)$ can be represented by a smooth semi-positive form and by Theorem \ref{B} we do not need to assume that $\lambda=\lambda_1(f)$ nor $\lambda$ is simple. In particular, we have the following result: 
 
\begin{thm}\label{zeros}
Let $f:X \to X$ be a 1-regular dominant rational map and $\lambda>1$ with $f^*\alpha=\lambda\alpha$ where $\alpha=c_1(L)$ and $L\to X$ is a semi-ample big line bundle. Then %for sufficiently divisible $m\gg0$ 
there exists $m>0$ and a pluripolar set $\mathcal{E}_m\subset \Bbb{P}H^0(X,L^{\otimes m})$ such that for every $s \not\in \mathcal{E}_m$ 
 $$\\\lambda^{-n}(f^n)^*(\frac{1}{m}[Z_s]) \to T_{\alpha}$$
in the sense of currents as $n\to\infty$.
\end{thm}  

  %If $X=\Bbb{P}^k$ and $L$ is the Hyperplane bundle $\mathcal{O}(1)$ then the corresponding result follows from \cite{RS} and \cite{FS} (see also \cite{DS06} for similar results).\\ \indent %We also show that the equidistribution result in the sense of \cite{RS} holds in a more general context: 
%Next, we consider the case where the algebra $$ R(L):=\bigoplus_{m\geq0}H^0(X,L^{\otimes m})$$ is finitely generated. 
\subsection{Equidistribution of Currents with Mild Singularities}
  Theorem \ref{A} indicates that if the class $\alpha$ can be represented by a positive closed current with mild singularities then any such current equidistributes towards the Green current.
 
%\begin{thm} \label{lelong}
%Let $f:X \to X$ be a dominant 1-regular meromorphic map. Assume that $\lambda:=\lambda_1(f)>1$ is a simple eigenvalue of $f^*$ and $f^*\alpha =\lambda \alpha$ with $\alpha\in\psef$. Furthermore, we assume that $\nu(\vmin,x)=0$ for every $x\in X.$ Then condition $(\star)$ of Theorem \ref{B} holds.
%Moreover, for every positive closed $(1,1)$ current $S\in \alpha$ such that the Lelong numbers, $\nu(S,x)=0$ for each $x\in X$ the sequence
%$\lambda^{-n}(f^n)^*S$ converges weakly to the Green current $T_{\alpha}.$ 
%\end{thm}
\begin{proof}[Proof of Theorem \ref{A}]
Let $S\in \alpha$ be a positive closed $(1,1)$ current such that $\nu(S,x)=0$ for every $x\in X.$ Then by $dd^c$ lemma we write $S=\theta + dd^c\varphi$. By Theorem \ref{B} and Corollary \ref{UI}, it is enough to show that $Vol(\{\varphi \circ f^n<-t\lambda^n\}) \to 0$ as $n\to \infty$ for every $t>0.$ Since 
$$f^n(\{\varphi \circ f^n<-t\lambda^n\}) \subset \{\varphi<-t\lambda^n\}$$ the claim follows from volume estimates as in proof of Proposition \ref{normal}.
 %On the other hand, since  the Lelong numbers $\nu( \varphi,x)=0$ for every $x \in X$ it follows from [Theorem 3.1, \cite{Ki}] that for every $A>0$ there exists $B>0$ such that  
%$$(C_1Vol(\{\varphi \circ f^n<-t\lambda^n\}))^{C_2\lambda^n}\leq B \exp(-At\lambda^n)$$
%where $C_1,C_2>0$ as in Theorem \ref{vol}. Since $A>0$ can be chosen arbitrarily large the assertion follows. 
\end{proof}
Theorem \ref{A} extends \cite[Theorem 1.4]{G03} which was obtained in the special  case when $X=\Bbb{P}^k$ and and $\alpha=\{\omega_{FS}\}$. In fact, if $\alpha$ can be represented by a positive closed current with identically zero Lelong numbers then $\alpha\in \nef$ (see \cite{Bou}). Thus, Theorem \ref{A} provides a new proof of a special case of \cite[Theorem A]{DG}.
Recall that if $\alpha \in \nef \cap H^{1,1}_{big}(X,R)$ then $\nu(\Tmin,x) \equiv 0$ on X. However, without the big assumption this is no longer true even in dimension 2. There exists a ruled surface X over an elliptic curve such that X contains an irreducible curve $C$ with the following property: $\{C\} \in \nef$ but $\{C\}$ contains only one positive closed (1,1) current  $[C]$ namely, the current of integration along C (see \cite{DPS} for details). On the other hand, even if the class $\alpha\in \bigc \cap \nef$ the function $\vmin$ might have a non-empty polar set (see \cite[Example 5.8]{BEGZ}).
\subsection{Equidistribution in Energy Classes}
%In this section, we consider some cases where the invariant class is merely big.
\begin{prop}\label{cap}
Let $f$ be as in Theorem \ref{B}. We assume that $\lambda:=\lambda_1(f)>1$ is simple and $\alpha \in \bigc$. If $\varphi \in \psh$ such that 
$$t^k Cap_{\theta}(\varphi-\vmin<-t)\to 0 \ as \ t\to +\infty$$
then $\frac{1}{\lambda^n}(f^n)^*(\theta+dd^c\varphi) \to T_{\alpha}$ in the sense of currents as $n\to \infty$.
\end{prop}
\begin{proof}
By shifting $\varphi$ we may assume that $\varphi\leq \vmin\leq 0.$ Then replacing $\phi$ by $\varphi-\vmin$ in the statement of Corollary \ref{UI} and using $\{\varphi-\vmin<-t\}\subset \{\varphi<-t\}$ one can show that the sequence $\{\frac{1}{\lambda^n}(\varphi-\vmin)\circ f^n\}$ is uniformly integrable. Thus, it is enough to show that for every $t>0$ 

$$Vol(\frac{1}{\lambda^n}(\varphi-\vmin)\circ f^n<-t)\rightarrow 0\ as\ n\rightarrow \infty.$$
Let $\nu_{\alpha}:=2 \sup_{T, x}\nu(T,x)>0$ where the $\sup$ is taken over all positive closed currents $T\in \alpha$ and $x \in X$. Note that $\nu_{\alpha}<\infty$ and depends only on the class $\alpha.$ Then by uniform version of Skoda's integrability theorem \cite{Ze} we have 
\begin{equation}\label{1}
\int _X \exp(-\nu_{\alpha}^{-1}\phi) dV \leq C_{\theta} \ for\ every\ \phi\in Psh(X,\theta) \ such\ that \ \sup_X\phi= 0
\end{equation}
 where $C_{\theta} >0$. 
Now, let us denote $K_t:=\{\varphi-\vmin<-t\}$ and apply (\ref{1}) with $\phi:=V^*_{K_t,\theta}-M_{\theta}(K_t)$ where $M_{\theta}(K_t):=\sup_XV^*_{K_t,\theta}.$ Then we get 
$$ \int_X\exp(-\nu_{\alpha}^{-1}V^*_{K_t,\theta}) dV \leq C_{\theta} \exp(- \nu_{\alpha}^{-1}M_{\theta}(K_t)) $$
Since $V^*_{K_t,\theta}\leq 0$ a.e. on $K_t$ with respect to Lebesgue measure \cite{GZ} we infer that 
$$Vol(\varphi -\vmin<-t) \leq C_{\theta}  \exp(- \nu_{\alpha}^{-1}M_{\theta}(K_t))$$
We need the following lemma (cf. \cite[Lemma 4.2]{BEGZ}):
\begin{lem}\label{l}
$$ Cap_{\theta}(K_t) \geq \frac{vol(\alpha)}{M_{\theta}(K_t)^k} $$
\end{lem}
\begin{proof}
Note that $M_{\theta}(K_t) \rightarrow \infty$ as $t\rightarrow \infty.$ Thus we may assume that $M_{\theta}(K_t)\geq 1.$ Now, we set $\sigma_t:= M_{\theta}(K_t)^{-1} V^*_{K_t, \theta}+(1-M_{\theta}(K_t)^{-1})\vmin-1$ then $\vmin-1\leq \sigma_t\leq \vmin$ and by definition of $Cap_{\theta}$ we get
$$Cap_{\theta}(K_t)\geq \int_X MA(\sigma_t) \geq M_{\theta}(K_t)^{-k}\int_X MA(V^*_{K_t,\theta})\geq \frac{vol(\alpha)}{ M_{\theta}(K_t)^k}$$
where the last inequality follows from the fact that $V^*_{K_t,\theta}$ is minimally singular ( cf \cite[Theorem 1.16  ]{BEGZ}).
\end{proof}
Thus, by Lemma \ref{l} we get 
$$Vol(\varphi-\vmin<-t) \leq C_{\theta} \exp(- \nu_{\alpha}^{-1}(\frac{vol(\alpha)}{Cap_{\theta}(K_t)})^{\frac{1}{k}})  $$
Now, for $t>0$ we let 
$$\Omega_{n,t}:=\{(\varphi-\vmin) \circ f^n<-t \lambda^n\}$$
then by Theorem \ref{vol} we have 
$$Vol(f^n(\Omega_{n,t}) \geq (C_1 Vol(\Omega_{n,t}))^{C_2\lambda^n}$$
hence,
$$0\leq Vol(\Omega_{n,t})\leq A\exp(-\frac{B}{t\lambda^n Cap_{\theta}(K_{t\lambda^n})^{\frac{1}{k}}})$$
for some $A,B>0.$ Thus, the assertion follows. 
\end{proof}

We stress that if $\varphi \in \psh$ such that $$t^k Cap_{\theta}(\varphi-\vmin<-t)\to 0 \ as \ t\to \infty$$ then $\int_X MA(\varphi)=vol(\alpha)$ that is $\varphi$ has full Monge-Amp\`ere. Indeed, letting $\varphi_t:=\max(\varphi,\vmin-t)$. Then $t^{-1}\varphi_t + (1-t^{-1})\vmin$ is a competitor for the Capacity for $t\geq 1$. Thus,
$$MA(\varphi_t)\leq t^k Cap_{\theta}$$
and we infer that $\int_{(\varphi-\vmin\leq-t)}MA(\varphi_t) \to 0$ as $t \to \infty$ and this implies that $\int_X MA(\varphi)=vol(\alpha).$\\ \indent
 In general, capacity of sublevel sets do not decay faster than $t^{-1}$ \cite{GZ1}. If $\Omega \subset \Bbb{C}^k$ is a bounded hyperconvex domain and $\varphi \in \mathcal{F}(\Omega)$ where $\mathcal{F}(\Omega)$ denotes the Cegrell class then $t^k Cap(\varphi<-t) \to 0$ as $t \to \infty$ (\cite{CKZ}). In the setting of compact K\"ahler manifolds, if $\alpha$ is a K\"ahler class or $\alpha \in NS(X)\cap\bigc$ semi-positive then for every $\varphi \in \psh$ with full Monge-Amp\`ere the Lelong numbers of $\varphi$ are identically zero on $X$ \cite{GZ1,BB}. Thus, by Theorem \ref{A} we have $\frac{1}{\lambda^n}(f^n)^*(\theta+dd^c\varphi) \to T_{\alpha}.$
\begin{question}
Let $f$ and $\alpha$ be as in Proposition \ref{cap}. Is is true that $$\frac{1}{\lambda^n}(f^n)^*(\theta+dd^c\varphi) \to T_{\alpha}$$ for every $\varphi\in \psh$ which has full Monge-Amp\`ere i.e. $\int_X MA(\varphi) =vol(\alpha)?$
\end{question}
In this direction, we prove the following result :
\begin{prop}
Let $f$ be as in Proposition \ref{cap}. If $\varphi \in \mathcal{E}_{\chi}(X,\theta)$ where $\chi$ is a convex weight such that $|\chi(-t)|=O(t^p)$ as $t\to \infty$ for some $p>k-1$ then $$\frac{1}{\lambda^n}(f^n)^*(\theta+dd^c\varphi) \to T_{\alpha}$$ in the sense of currents as $n\to \infty.$
\end{prop}
\begin{proof}
The proof follows from Proposition \ref{cap} and the Lemma \ref{tcap} which is variation of \cite[Lemma 5.1]{GZ1} together with the regularity result from \cite{BeDe}. 
\end{proof}
\begin{lem}\label{tcap}
If $\varphi \in \mathcal{E}_{\chi}(X,\theta)$ then there exists $C_{\varphi}>0$ such that $$Cap_{\theta}(\varphi-\vmin<-t) \leq \frac{C_{\varphi}}{t|\chi(-t)|}$$
for every $t>1.$ 
\end{lem}
\begin{proof}
Shifting $\varphi$ we may assume that $\varphi\leq \vmin$. Let $\psi \in \psh$ such that $\vmin-1\leq \psi \leq \vmin$ then for $t>1$
$$ \{\varphi<\vmin-2t\}\subset \{t^{-1}\varphi + (1-t^{-1})\vmin<\psi-1\}\subset \{\varphi<\vmin-t\}.$$
Since $\varphi$ has full Monge-Amp\`ere comparison principle \cite[Corollary 2.3]{BEGZ} yields 
\begin{eqnarray*}
\int_{\{\varphi-\vmin<-2t\}} MA(\psi) &\leq & \int_{\{\varphi-\vmin<-t\}} MA(t^{-1}\varphi + (1-t^{-1})\vmin) \\
& \leq & \int_{\{\varphi-\vmin<-2t\}}  MA(\vmin) + \sum_{j=1}^k \binom{k}{j} t^{-j}  \int_{\{\varphi-\vmin<-2t\}} \langle(\theta+dd^c\varphi)^{j}\wedge(\theta+dd^c\vmin)^{k-j}\rangle \\
& \leq & A\exp(-Bt) + \frac{C} {t} MA(\varphi)(\{\varphi-\vmin<-t\})
\end{eqnarray*}
where $A,B>0$ and the the last inequality follows from the fact $MA(\vmin)$ has $L^{\infty}$-density with respect to Lebesgue measure (\cite[Corollary 2.5]{BeDe}) and \cite[Proposition 2.8]{BEGZ}.  Then by Chebyshev's inequality we obtain
$$Cap(\varphi-\vmin<-t)\leq A\exp(-Bt)+ \frac{1}{t|\chi(-t)|}\int_X|\chi(\varphi-\vmin)|MA(\varphi)$$
since $\varphi \in \energyc$ the later integral is finite.
\end{proof}
%\begin{cor}
%Let $X$ be a compact K\"ahler surface and $f:X\to X$ be a dominant 1-regular meromorphic map. If $\lambda_1^2>\lambda_2$ then $\frac{1}{\lambda_1^n} (f^n)^*(\theta+ dd^c\varphi) \to T_{\alpha}$ for every $\varphi \in \mathcal{E}^{p}(X,\theta)$ with $p>1.$
%\end{cor}
%\begin{rem}
 %We remark that the condition $$\int_X|\chi(\varphi-\vmin)|MA(\varphi)<\infty$$
 %is not enough to conclude that $\theta+dd^c\varphi$ equidistributes towards the Green current. Indeed, let $f:\Bbb{P}^1\to \Bbb{P}^1$ be a holomorphic map of degree $\lambda\geq 2$ with a totally invariant point $p$. Then $\delta_p=\omega_{FS}+dd^c\varphi_p$ is a positive closed current with $\int_{\Bbb{P}^1}(\varphi_p)MA(\delta_p)=0$ but $\frac{1}{\lambda^n}(f^n)^*\delta_p\not\to \mu_f$.
 %\end{rem}

\subsection{Minimally Separating Maps}\label{min}
Let $f:X\to X$ be a bimeromorphic map. We denote the indeterminacy locus of $f^{-1}$ by $I^-$ and we set $\mathcal{I_1^-}=\bigcup_{n=1}^{\infty}f^{n}(I^-).$ %One can also define the analogues sets for the mapping $f$. 
Following \cite{Di}, we say that $f$ is \textit{minimally separating} if $\mathcal{I_1^+} \cap \mathcal{I_1^-}=\emptyset.$ 

\begin{thm}\label{SM}
Let $f:X\to X$ be a minimally separating map. We assume that $\lambda=\lambda_1(f)>1$ is simple and the corresponding normalized eigenvector $\alpha \in \bigc$. If $\mathcal{I_1^-} \subset Amp(\alpha)$ then condition $(\star)$ holds. Moreover, $\lim\frac{1}{\lambda^n}(f^n)^*\eta=cT_{\alpha}$ for every smooth closed $(1,1)$ form $\eta$ on $X$ where $c>0$ depends only on the class $\{\eta\}$.  
\end{thm}
\begin{proof}
It is easy to see that a minimally separating bimeromorphic map is always 1-regular. We fix a smooth representative $\theta\in \alpha$ and let $T_K=\theta +dd^c \varphi_K$ be a K\"ahler current such that $Amp(\alpha)= X \setminus E_+(T_K).$ We first show that $\frac{1}{\lambda^n}\varphi_K\circ f^n \to 0$ in $L^1(X)$ which clearly implies the condition $(\star).$ By Corollary \ref{UI} it is enough to show that $Vol(\Omega_{n,t})\to 0$ as $n \to \infty$ where 
$$\Omega_{n,t}:=\{x\in X: \frac{1}{\lambda^n}\varphi_K\circ f^n(x)<-t\}.$$ Assuming the contrary we will drive a contradiction: suppose that there exists $r>0$ and $n_k \to \infty$ such that $r\leq Vol(\Omega_{n_k,t})$. Renumbering the sequence we may assume that $n_k=n$. We fix $n$ large enough such that $\exp(-C\lambda^n)<r$ where $C>0$ will be defined later. Since $f$ is minimally separating and $E_+(T_K)$ is an analytic set by hypothesis there exists $\epsilon=\epsilon(n)>0$ small enough and open neighborhoods  $U_{\epsilon}(E_+(T_K))$ and $U_{\epsilon}((I^-) \cup f(I^-) \dots \cup f^{n-1}(I^-))$ of $E_+(T_K)$ and $(I^-) \cup f(I^-) \dots \cup f^{n-1}(I^-)$ respectively such that $$U_{\epsilon}(E_+(T_K))\ \cap\ U_{\epsilon}((I^-) \cup f(I^-) \dots \cup f^{n-1}(I^-))=\emptyset.$$
\begin{lem} \label{jac}
There exists $\delta=\delta(\epsilon)>0$ such that if $p\in X\setminus I_f$ with $$|Jac_{\omega}(f)(p)|^2<\delta$$ then $f(p) \in U_{\epsilon}(I^-).$ 
\end{lem}
\begin{proof}
Let $B(f(p),s)$ denote a small ball around $f(p).$ Assume that $f(p) \not\in U_{\epsilon}( I^-)$ then 
\begin{eqnarray*}
 |Jac_{\omega}(f)(p)|^2 &=& \lim_{s \to 0}\frac{Vol(B(f(p),s)}{Vol(f^{-1}(B(f(p),s)))}\\
 &\geq& \lim_{s\to 0}(\frac{s}{\max(\{dist(f^{-1}(x),p): x\in B(f(p),s)\})})^{2k}\\
 &\geq& \frac{1}{||Df^{-1}(p)||^{2k}} \\
 &\geq & dist(f(p), I^-)^q\\ 
 &\geq& \epsilon^q = \delta
 \end{eqnarray*}
 where the fourth inequality follows from \cite[Lemma 2.1]{DD} and $q>0$ independent of $p.$ 
 \end{proof}
 Now we choose $m$ large such that $$f^{n+m}(\Omega_{n+m,t})\subset \{\varphi_K<-t\lambda^{n+m}\} \subset U_{\epsilon}(E_+(T_K))$$ where $$\Omega_{n,t}:=\{\varphi_K \circ f^n<-t \lambda^n\}.$$ Since $U_{\epsilon}(E_+(T_K)) \cap U_{\epsilon}((I^-) \cup f(I^-) \dots \cup f^{n-1}(I^-))=\emptyset$ it follows from Theorem \ref{vol} and Lemma \ref{jac} that
 $$Vol(f^{n+m}(\Omega_{n+m,t})) \geq \delta^n(C_1Vol(\Omega_{n+m,t}))^{C_2\lambda^{m}}.$$
 On the other hand, since $f^{n+m}(\Omega_{n+m,t})\subset \{\varphi_K<-t\lambda^{n+m}\}$ there exists constants $A,B>0$ independent of $n$ and $m$ such that 
 $$A\exp(-B\lambda^{n+m}) \geq Vol(\varphi_K<-t\lambda^{n+m}).$$
 Combining these two inequalities and letting $m\to \infty$ we get
 $$r \leq \liminf_{m\to \infty} Vol(\Omega_{m,t})\leq \exp(-C\lambda^n) $$
where $C=\frac{B}{C_2}>0$ is the positive constant. This arrises a contradiction and we conclude that $\frac{1}{\lambda^n}\varphi_K\circ f^n \to 0$ in $L^1(X)$.\\ \indent
Next,  we fix a K\"ahler form $\beta$ such that $T_K\geq \beta$. Any limit point $S$ of the sequence $\{\lambda^{-n}(f^n)^*\beta\}$ is a positive closed current in $c\alpha$ for some $c>0$ depending only on $\{\beta\}.$ This implies that $S\leq cT_{\alpha}$. Hence, by Theorem \ref{B} (2) we get $S= cT_{\alpha}$.\\ \indent 
 Since $X$ is compact for any two K\"ahler forms $\omega_1, \omega_2$ on $X$ there exists a constant $C>0$ such that $\frac{1}{C}\omega_1\leq \omega_2\leq C\omega_1$ and this implies the weak convergence under pull-backs towards Green current for any K\"ahler form. Finally, any smooth form $\eta$ can be written as a difference of two K\"ahler forms hence the assertion follows.
 \end{proof}
The following result is a consequence of Crofton's formula applied as in Theorem \ref{finite}: 
\begin{cor}
Let $f$ be as in Theorem \ref{SM}. If X is projective then for Lebesgue almost every hyperplane section $H$ $$\frac{1}{\lambda^n}(f^n)^*[H]\to cT_{\alpha}$$
where $c>0$ depends only on the imbedding $X \hookrightarrow \Bbb{P}^N.$  
\end{cor}

\section{Proof of Theorem \ref{class}}
%We stress that if $X$ is a compact K\"ahler surface then a class $\alpha \in \nef$ has numerical dimension one if and only if $\alpha$ is not big.  

%\begin{thm} \label{class}
%Let $X$ be a compact K\"ahler surface and $f:X\to X$ be a dominant meromorphic map. If $\lambda_1(f)>\lambda_2(f)$ and $\alpha_f\in \bigc$ then $X$ is rational.
%\end{thm}
\begin{proof}[Proof of Theorem \ref{class}]
Since $f^*$ preserves the nef cone we may and we do assume that $\alpha \in \nef.$ By $\lambda_1:=\lambda_1(f)>d_{top}(f),$ it follows from \cite[Theorem 2.14]{G} that either $kod(X)=0\ \text{or}\ X$ is rational. \\ \indent
If $kod(X)=0$ and $X$ is minimal surface then by Kodaira-Enrique classification of compact complex surfaces \cite{BPV}, $K_X$ is a nef-divisor and $12K_X\equiv0$. Then by Riemann-Hurwitz formula \cite{GH} we have 
$$K_X=R_f+f^*K_X$$
where $R_f$ denotes the ramification divisor. Thus, we deduce that $R_f=0.$ This implies that the exceptional locus $\mathcal{E}_f=\emptyset.$ Then the push-pull formula \cite[Theorem 3.3]{DF} yields that $\alpha_f^2=0.$ Since $\alpha \in \nef$ by \cite{DP} we have $vol(\alpha_f)=\alpha_f^2.$ Therefore, $\alpha_f$ is not big.\\ \indent
If $X$ is not minimal then there exists a proper holomorphic map (induced by a finite sequence of blow-ups) $\pi: X \to X'$ where $X'$ is minimal. Moreover, $f$ induces a meromorphic map $g:X' \to X'$ where $g=\pi \circ f \circ \pi^{-1}$. Since the dynamical degrees are bimeromorphic invariants we obtain that $\lambda_1=\lambda_1(g)>d_{top}(g)=d_{top}(f)$. Thus, by \cite[Theorem 0.3]{DF}, $\lambda_1$ is a simple eigenvalue of $g^*.$ We denote the invariant class $\alpha_g\in \nef$ which is normalized by $\langle \alpha_g,\omega_{X'} \rangle=1$. Moreover, we claim that $\alpha_g=\pi_*\alpha_f$ up to a constant. Indeed, by push-pull formula \cite[Theorem 3.3]{DF} we have
\begin{eqnarray*}
\lambda_1^{-n}(g^n)^*(\pi_*\alpha_f) & = & \lambda_1^{-n} \pi_*(f^n)^*(\pi^*\pi_*\alpha_f)\\
& = & \lambda_1^{-n} \pi_* (f^n)^*(\alpha_f+E) \\
& = & \pi_*\alpha_f + \pi_*( \lambda_1^{-n} (f^n)^*E)
\end{eqnarray*}
where $E$ is the class of an exceptional effective divisor supported in $\mathcal{E}(\pi).$ Since $f$ is 1-regular and $E\in \psef$ we have $$\lim_{n\to \infty}\lambda_1^{-n}(f^n)^*E=c_1\alpha_f$$ for some $c_1\geq0.$ On the other hand, since $X'$ is minimal and $kod(X')=0,$ by above argument $\mathcal{E}_g=\emptyset$ hence, $g$ is 1-regular and from $\pi_*\alpha_f\in H^{1,1}_{psef}(X',\Bbb{R})$ we infer that $$\lim_{n\to \infty}\lambda^{-n}_1(g^n)^*(\pi_*\alpha_f)=c_2\alpha_g$$ for some $c_2>0.$ Hence $\alpha_g=C\pi_*\alpha_f$ for some $C>0.$\\ \indent
 Finally, \cite[Corollary 3.4]{DF} applied to $\pi^{-1}$ we get $$\langle\alpha_g,\alpha_g\rangle= \langle \pi_*\alpha_f, \pi_*\alpha_f\rangle\geq \langle\alpha_f,\alpha_f\rangle=vol(\alpha).$$ Thus, by above argument $\alpha_f^2=0$ and $\alpha_f$ is not big.  \\ \indent
%If $kod(X)=-\infty$ and $X$ is rational then we are done. If $X$ is not rational then it is a ruled surface with a nonrational base \cite{BPV}. Since $f$ preserves the fibration by \cite[Theorem 2.4]{G} we conclude that $\lambda_2\geq \lambda_1$ which contradicts our assumption.
\end{proof}

 \bibliographystyle{amsalpha}

\bibliography{biblio}

\end{document}